\documentclass{amsart}

\usepackage{amsmath, amsfonts,amssymb,amsthm,array}
\usepackage{graphicx}
\usepackage{thmtools,thm-restate}
\usepackage{enumitem}
\usepackage{float}
\usepackage{tikz}
\usepackage{comment}
\usetikzlibrary{positioning,decorations.pathmorphing,arrows.meta}
\tikzset{main node/.style={circle,fill=black,draw,minimum width=4pt,inner sep=0pt}}
\newcounter{counter}

\newtheorem{theorem}[counter]{Theorem}
\newtheorem{lemma}[counter]{Lemma}

\theoremstyle{definition}
\newtheorem{definition}[counter]{Definition}

\newenvironment{claim}[1]{\vspace{0.2cm}\par\noindent\emph{Claim.}\space#1}{}

\newcommand{\N}{\mathbb{N}}

\newcommand{\set}[1]{\{#1\}}
\newcommand{\cf}{\textup{cf}}

\newcommand{\parentheses}[1]{{\left( {#1} \right)}}

\newcommand{\p}{\parentheses}

\newcommand{\Set}[1]{{\left\lbrace {#1} \right\rbrace}}

\def\set#1:#2{\Set{{#1} \colon {#2}}}

\begin{document}

\author{Carl B\"urger}
\author{Max Pitz}
\address{University of Hamburg, Department of Mathematics, Bundesstra{\ss}e 55 (Geomatikum), 20146 Hamburg, Germany}
\email{carl.buerger@uni-hamburg.de, max.pitz@uni-hamburg.de}

\title[Partitioning bipartite graphs into monochromatic paths]{Partitioning edge-coloured infinite complete bipartite graphs into monochromatic paths}

\begin{abstract}
In 1978, Richard Rado showed that every edge-coloured complete graph of countably infinite order can be partitioned into monochromatic paths of different colours. He asked whether this remains true for uncountable complete graphs and a notion of \emph{generalised paths}. In 2016, Daniel Soukup answered this in the affirmative and conjectured that a similar result should hold for complete bipartite graphs with bipartition classes of the same infinite cardinality, namely that every such graph edge-coloured with $r$ colours can be partitioned into $2r-1$ monochromatic generalised paths with each colour being used at most twice. 

In the present paper, we give an affirmative answer to Soukup's conjecture.
\end{abstract}

\maketitle

\section{Introduction}

Throughout this paper, the term \emph{colouring} always refers to edge colourings of graphs with finitely many colours. 

In the 1970s, Erd\H{o}s proved (unpublished) that every $2$-coloured complete graph of countably infinite order, i.e.\ every $2$-coloured $K_{\aleph_0}$, can be partitioned into monochromatic paths of different colours, where `path' means either a finite path or a one-way infinite ray. Rado subsequently extended Erd\H{o}s result to any finite number of colours \cite[Theorem 2]{R}. 

In the same paper, Rado then asked whether a similar result holds for all infinite complete graphs, even the uncountable ones. Clearly, it is not possible to partition such a graph into finitely many `usual' paths, as graph-theoretic paths and rays are inherently countable. Hence, Rado introduced the following notion of \emph{generalised path}: A generalised path is a graph $P$ together with a well-order $\prec$ on $V(P)$ (called the \emph{path order} on $P$) satisfying that the set $\{w\in N(v) \colon w\prec v\}$ of down-neighbours of $v$ is cofinal below $v$ for every vertex $v\in V(P)$, i.e.\ for every $v' \prec v$ there is a neighbour $w$ of $v$ with $v' \preceq w \prec v$ (cf. Figure \ref{fig: path}).\vspace{0,5cm}
\begin{figure}[h]
\begin{center}
\def\svgwidth{10cm}
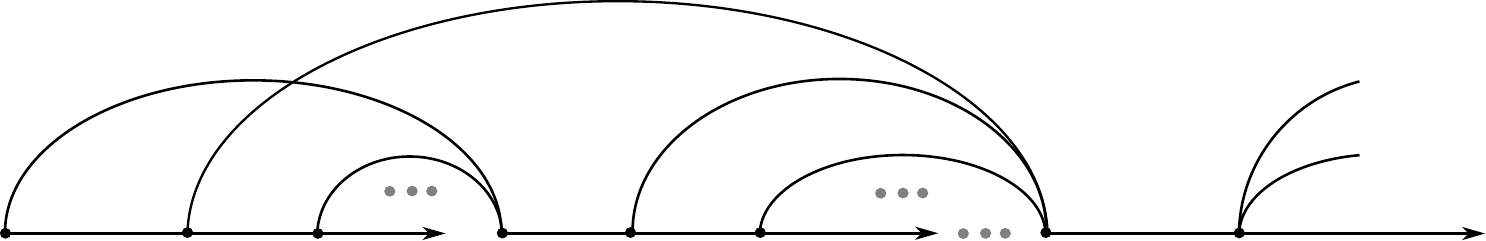
\end{center}
\caption{A generalised path.}
\label{fig: path}
\end{figure} 

In particular, every successor element is adjacent to its predecessor in the well-order. Calling such a graph $P$ a `generalised path' is justified by the fact that between any two vertices $v \prec w$ of $P$ there exists a \emph{finite} path from $v$ to $w$ strictly increasing with respect to $\prec$, see e.g.\ \cite[Observation 5.2]{ESSS17}. If the situation is clear, we write $P$ instead of $(P,\prec)$ and treat $P$ as a graph. By $\Lambda(P,\prec)=\Lambda(P)$ we denote the limit elements of the well-order $(P,\prec)$. If necessary, the path-order $\prec$ on $V(P)$ will be referred to as $\prec_P$. If $v,v' \in P$, then we denote by $(v,v')$ and $[v,v']$ the open and closed intervals with respect to $\prec$, and by $[v,v+\omega)$ the ray of $P$ starting at $v$ compatible with the path order. Note that a one-way infinite ray can be viewed quite naturally as a generalised path of order type $\omega$, and conversely, every generalised path of order type $\omega$ contains a spanning one-way infinite ray. Thus, partitioning a graph into monochromatic generalised path of order type $\omega$ is equivalent to partitioning it into monochromatic rays.

From now on, the term \emph{path} is used in the extended sense of a generalised path.

Elekes, Soukup, Soukup and Szentmikl\`{o}ssy \cite{ESSS17} have recently answered a special case of Rado's question for $\aleph_1$-sized complete graphs and two colours in the affirmative. Shortly after, Soukup \cite{S17} gave a complete answer to Rado's question for any finite number of colours and complete graphs of arbitrary infinite cardinality.

\begin{restatable}[Soukup, {\cite[Theorem 7.1]{S17}}]{theorem}{CompleteCase}\label{thm: decomposition thm for complete graph}
Let $r$ be a positive integer. Every $r$-edge-coloured complete graph of infinite order can be partitioned into monochromatic generalised paths of different colours. 
\end{restatable} 

In \cite[Conjecture 8.1]{S17}, Soukup conjectures that a similar result holds for complete bipartite graphs, namely that every $r$-coloured complete bipartite graph with bipartition classes of cardinality $\kappa\ge \aleph_0$ can be partitioned into $2r-1$ monochromatic generalised paths, and has proven his conjecture in the countable case $\kappa = \aleph_0$ \cite[Theorem 2.4.1]{S15}. If true, this bound would be best possible in the sense that there are $r$-colourings of $K_{\kappa,\kappa}$ for which the graph cannot be partitioned into $2r-2$ monochromatic paths, see \cite[Theorem 2.4.1]{S15}.

We remark that Soukup's conjecture is inspired by the corresponding conjecture in the finite case, due to Prokovskiy \cite[Conjecture 4.5]{P14}. In contrast to the infinite case, the finite conjecture is only known for two colours \cite[p. 169 (footnote)]{GL73}.

The main result of this paper is to prove Soukup's conjecture for all uncountably cardinalities and any (finite) number of colours.

\begin{restatable}{theorem}{MainResult}\label{thm: main result}
Let $r$ be a positive integer. Every $r$-edge-coloured complete bipartite graph with bipartition classes of the same infinite cardinality can be partitioned into $2r-1$ monochromatic generalised paths with each colour being used at most twice. 
\end{restatable} 

The first uncountable case of Theorem~\ref{thm: main result}, where the bipartition classes have size $\aleph_1$, was proved by the first author in his Master's thesis \cite{C}. In this paper, we extend these ideas to give a proof for all uncountable cardinalities.

Our proof relies on the methods developed by Soukup in his original paper \cite{S17}. However, we re-introduce in this paper the new, helpful notion of \emph{$X$-robust} paths from \cite{C}---generalised paths which are resistant against the deletion of vertices from $X$. After introducing such paths, we will state in Section~\ref{sec2} three high-level results relying on this new notion, and then provide a proof of Theorem~\ref{thm: main result} from these auxiliary results. In fact, our discussion will also lead to a new, conceptually simpler closing argument for a proof of Soukup's Theorem~\ref{thm: decomposition thm for complete graph}.

In Sections~\ref{sec3} and \ref{sec4}, we then provide proofs of the auxiliary results. For the second of these auxiliary results, to be proved in Section~\ref{sec3}, we need to strengthen a  result by Soukup \cite[\S5]{S17} to give the statement that any edge-coloured complete bipartite graph with bipartition classes $(A,B)$ of cardinality $\kappa>\aleph_0$ contains a monochromatic path $P$ of order type $\kappa$ in colour $k$ (say) covering a large subsets $X \subseteq A$ which itself is $\kappa$-star-linked in colour $k$, where it is precisely the \emph{$\kappa$-star-linked-property} (to be defined below) which is new. We remark that while our statement is slightly stronger, our proof very much relies on Soukup's proof \cite[\S5]{S17} and does not give an independent proof of Soukup's result. A discussion how one obtains the strengthened version of Soukup's result is provided in Section~\ref{sec_final}.

Finally, in Section~\ref{sec4} we prove our third auxiliary result. This part contains a crucial new idea how to directly construct an $X$-robust path $Q$ of order type $\kappa>\aleph_0$ with $X\in [V(Q)]^\kappa$ from a given generalised path $P$ with the star-linked property as above, using nothing but countable combinatorics and avoiding intricate set theoretical arguments using elementary submodels as employed in \cite{S17} and \cite{C}. 

\subsection*{Notation} For graph theoretic notation we follow the text book \emph{Graph Theory} \cite{D} by Diestel. For a natural number $n\in \N$ we write $[n] = \{1,2,\ldots, n \}$ and if $m\le n$, we write $[m,n]=\{m,m+1,\dots,n\}$. 
  
Let $G=(V,E)$ be a graph, $r\ge 1$ and $k\in [r]$. An \emph{$r$-edge-colouring} (or simply $r$-\emph{colouring}) of $G$ is a map $c\colon E\rightarrow [r]$. A path $P\subseteq G$ is \emph{monochromatic} (in colour $k$ with regard to the colouring $c$) if $P$ is also a path in the graph induced by the edges of colour $k$, i.e. if $P$ is a path in $G[\bigcup c^{-1}(k)]$. More generally, suppose that $\mathcal{P}$ is a graph property. We say that $G$ \emph{has property $\mathcal{P}$ in colour $k$} if $G[\bigcup c^{-1}(k)]$ has property $\mathcal{P}$. For a vertex $v$ of $G$ we write $N(v,k)$ for the neighbourhood of $v$ in $G[\bigcup c^{-1}(k)]$. As a shorthand, we also write $N(v,{\neq}k) := N(v) \setminus N(v,k)$ for the neighbourhood of $v$ in all colours but $k$. Let $A \subseteq V$. The common neighbourhood $\bigcap\{N(v)\colon v\in A\}$ of vertices in $A$ is written as $N[A]$. The common neighbourhood of $A$ in colour $k$ is written as $N[A,k]$. For a cardinal $\kappa$, we say that $A$ is  \emph{$\kappa$-star-linked in $B$}, if $N[F]\cap B$ has cardinality $\kappa$ for every finite $F\subseteq A$. In the case where $B=V(G)$ we simply say that $A$ is $\kappa$-star-linked. 

When talking about \emph{partitions of $G$} we always mean vertex partitions and we allow empty partition classes. If $A,B\subseteq V(G)$ are disjoint sets of vertices, then $G[A,B]$ denotes the bipartite graph on $A \cup B$ given by all the edges between $A$ and~$B$.

For a set $X$, we write $[X]^\kappa = \set{Y \subseteq X}:{|Y| = \kappa}$ and $[X]^{{<}\kappa} = \set{Y \subseteq X}:{|Y| < \kappa}$.


\section{A high-level proof of the main result}
\label{sec2}

The aim of this section is to give an overview of the proof of Theorem \ref{thm: main result}. We shall start with a rough idea, inspired by Soukup's work in \cite[Theorem 7.1]{S17}. After that, we present three main ingredients for our proof of Theorem \ref{thm: main result}: Lemma \ref{lemma: covering bipartition classes}, Lemma \ref{lemma: a path and a star-linked set} and Lemma \ref{lemma: constructing robust paths}. For the moment, we will skip the latter two and discuss them below in Section 3 and Section 4. We conclude this section with a proof of Theorem \ref{thm: main result} and a proof of Theorem \ref{thm: decomposition thm for complete graph}---also based on the three lemmas.

\subsection{A rough outline} First, let us have a look at an important idea in Soukup's proof of Theorem~\ref{thm: decomposition thm for complete graph}. In \cite[Lemma 4.6]{S17}, Soukup provides some \emph{conditions} 
which guarantee the existence of a spanning generalised path in a graph. Let us refer to these conditions by $(\dagger)$. Let $\kappa$ be an infinite cardinal and $G=(V,E)$ the complete graph of order $\kappa$. Suppose that the edges of $G$ are coloured with $r\ge 1$ many colours. In \cite[Claim~7.1.2]{S17}, Soukup shows that one can find sets $X \subseteq W\subseteq V$ and a colour $k\in [r]$, such that \begin{enumerate}
\item $G[W\setminus X']$ satisfies $(\dagger)$ in colour $k$ for every $X'\subseteq X$, and 
\item $V\setminus W$ is covered by disjoint monochromatic paths of different colours not equal to $k$ in the graph $G[V\setminus W,X]$.
\end{enumerate}
Once such $W,X$ and $k$ are found, we just have to find $r-1$ disjoint monochromatic paths of different colours $\neq k$ covering $V\setminus W$ in $G[V\setminus W,X]$ as in (2), let $X' \subset X$ be the vertices of $X$ covered by these $r-1$ paths, and apply (1) to guarantee the existence of a monochromatic path in colour~$k$ disjoint from all previous ones and covering the remaining vertices. 

Whilst it is difficult to work with the conditions from $(\dagger)$ in the bipartite setting directly, the use of $(\dagger)$ in (1) motivates the following definition: 

\begin{definition}
Let $P$ be a path and $X\subseteq V(P)$. We say that $P$ is \emph{$X$-robust} iff for every $X'\subseteq X$ the graph $P-X'$ admits a well-order for which $P-X'$ is a path of the same order type as $P$. 
\end{definition} 

Our strategy for the bipartite case can now by summarised as follows. Let $\kappa$ be an infinite cardinal and $G=(V,E)$ the complete bipartite graph with bipartition classes of cardinality $\kappa$, where the edges of $G$ are coloured with $r\ge 1$ many colours. Assume that we find $X \subseteq W\subseteq V$ and a colour $k\in [r]$, such that 
\begin{itemize}
\item[$(1')$] $G[W]$ has a spanning $X$-robust path in colour $k$, and
\item[$(2')$] $V\setminus W$ is covered by $2r-2$ disjoint monochromatic paths in the graph $G[V\setminus W,X]$ in colours not equal to $k$ with every colour appearing at most twice.
\end{itemize}
Then it is clear that we can complete a proof of Theorem \ref{thm: main result} in a similar way as above.

\subsection{The three ingredients} 
To prove our main theorem, we shall need the following three ingredients.

\begin{theorem}[{Soukup, \cite[Thm~6.2]{S17}}]\label{lemma: covering bipartition classes}
Let $G$ be an infinite bipartite graph with bipartition classes $A$ and $B$. Suppose that $|A|\le |B|$ and that $|B\setminus N(a)|<|B|$ for every vertex $a\in A$. Then for every finite edge colouring of $G$, there are disjoint monochromatic paths of different colours in $G$ covering $A$. 
\end{theorem}

\begin{proof} By the argument in \cite[p. 271, l. 17-20]{S17}, the theorem follows from \cite[Theorem 6.2]{S17}. \end{proof}

The next main lemma, which is a strengthening of a similar result by Soukup \cite[\S5]{S17}, helps to find a monochromatic path $P$ which has some desirable additional properties. 

\begin{restatable}{lemma}{SecondMainLemma}\label{lemma: a path and a star-linked set}
 Let $\kappa$ be an infinite cardinal and $G$ the complete bipartite graph with bipartition classes $A,B$ both of cardinality $\kappa$. Suppose that $c\colon E(G)\rightarrow [r]$ is a colouring of $G$ with $r\ge 1$ many colours. Then there are disjoint sets $A_1,A_2\in[A]^\kappa$, $B_1,B_2\in [B]^\kappa$ such that (up to renaming the colours): \begin{itemize}
 \item $G[A_1,B_1]$  has a spanning path $P$ of order type $\kappa$ in colour~$1$ all whose limits are contained in $B_1$, and 
 \item $A_1\sqcup A_2$ is $\kappa$-star-linked in $B_2$ in colour $1$. (cf. Figure \ref{fig: A path and a star-linked set})
 \end{itemize} 
\end{restatable}

\begin{figure}[t]
\begin{center}
\def\svgwidth{6cm}
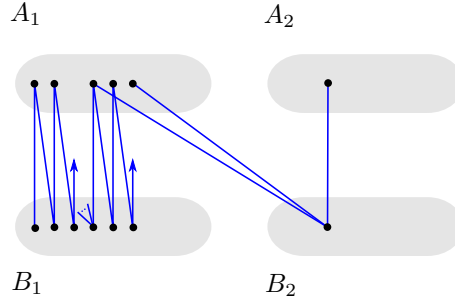
\end{center}
\vspace{0,5cm}
\caption{The colour $1$ is indicated blue in the figure. }
\label{fig: A path and a star-linked set}
\end{figure}

Our final ingredient converts the path $P$ from above into a new path $Q$ that has two additional properties: first, $Q$ will be $X$-robust for some large $X$, and secondly, $Q$ will be able to additionally cover certain highly inseparable sets of vertices.   

\begin{definition}[{cf. Diestel, \cite[ p. 354]{D}}]
Let $G$ be a graph and $\kappa$ a cardinal. A set $U\subseteq V(G)$ of vertices is \emph{${<}\kappa$-inseparable} if distinct vertices $v,w\in U$ cannot be separated by less than $\kappa$ many vertices, i.e. $v$ and $w$ are contained in the same component of $G-W$ for every $W\in [V(G)\setminus\{v,w\}]^{<\kappa}$. 
\end{definition}

\begin{restatable}{lemma}{ThirdMainLemma}\label{lemma: constructing robust paths}
Let $\kappa$ be an uncountable cardinal and $G$ a bipartite graph with bipartition classes $A,B$ both of size $\kappa$. Suppose there are disjoint sets $A_1,A_2 \in [A]^\kappa$, $B_1,B_2\in [B]^\kappa$ such that \begin{itemize}
\item $G[A_1,B_1]$  has a spanning path $P$ of order type $\kappa$ with $\Lambda(P) \subset B_1$, and 
 \item $A_1\sqcup A_2$ is $\kappa$-star-linked in $B_2$.
\end{itemize}  
Then there is a set $X\in [A_2]^\kappa$ and an $X$-robust path $Q$ covering $A_1\sqcup A_2$ with $\Lambda(P) = \Lambda(Q)$. Moreover, if $C\subseteq (A\setminus A_1)\sqcup (B\setminus \Lambda)$ covers $A_2$ and is ${<}\kappa$-inseparable in $G[A\setminus A_1, B\setminus \Lambda]$, then $Q$ can be chosen to cover $C$. 
\end{restatable}

\subsection{Combining the ingredients} Our three main ingredients can be applied to yield a proof of Theorem~\ref{thm: main result} as follows:

\MainResult*

\begin{proof}[Proof of Theorem \ref{thm: main result}]
Let $\kappa$ be an infinite cardinal and $G$ the complete bipartite graph with bipartition classes $A,B$ both of cardinality $\kappa$. Suppose that $c\colon E(G)\rightarrow [r]$ is a colouring of $G$. Since the countable case has been solved in \cite[Theorem 2.4.1]{S15} already, we may assume that $\kappa$ is uncountable. 

We will construct a partition $\mathcal{A}=\{A_1,\dots, A_4\}$ of $A$ and a partition $\mathcal{B}=\{B_1,\dots,B_4\}$ of $B$ such that, up to renaming the colours,
\begin{enumerate}
[label={\upshape(\roman{*})}]
\item\label{item1} $G[A_1,B_1]$ has a spanning path $P$ of order type $\kappa$ in colour~$1$ with $\Lambda(P) \subset B_1$, and $|A_2|=\kappa$,
 \item\label{item2} $A_1\sqcup A_2$ is $\kappa$-star-linked in $B_2$ in colour~$1$,
\item\label{item3} $A_2\sqcup A_3$ is ${<}\kappa$-inseparable in $G[A_2\sqcup A_3,B_2\sqcup B_3]$ in colour $1$, and
\item\label{item4} $A_4\sqcup B_4$ can be partitioned into $r-1$ monochromatic paths $P_2,\dots, P_{r}$ in $G[A_4,B_4]$ with distinct colours in $[2,r]$. 
\end{enumerate}

\begin{figure}[H]
\begin{center}
\def\svgwidth{12cm}
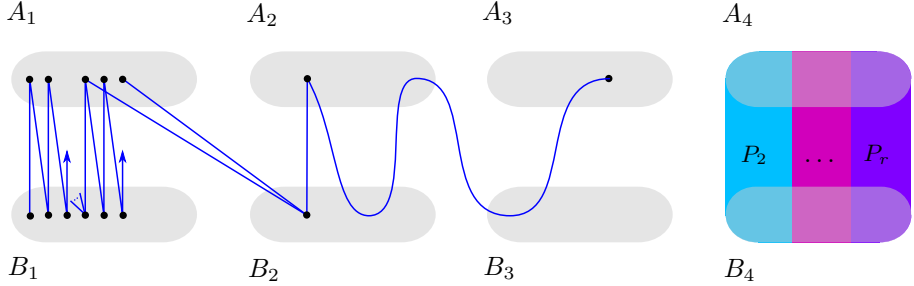
\end{center}
\vspace{0,5cm}
\caption{The colour $1$ is indicated blue in the figure. }
\label{fig: main proof 1}
\end{figure}

Let us first see how to complete the proof with these partitions established: 

Let $C$ be the set of vertices with $A\cap C= A_2\cup A_3$ and where $B\cap C$ consist of those vertices in $B\setminus (\Lambda \cup B_4)$ that send $\kappa$ many edges in colour~$1$ to $A_2$, and observe that \ref{item3} implies that $C$ is ${<}\kappa$-inseparable in $G[A \setminus (A_1 \cup A_4),B \setminus (\Lambda \cup B_4)]$ in colour~$1$. Hence, by \ref{item1} and \ref{item2}, we may apply Lemma~\ref{lemma: constructing robust paths} in $G[A\setminus A_4,B\setminus B_4]$ induced by the edge of colour $1$ in order to obtain a set $X\in [A_2]^\kappa$ and an $X$-robust, monochromatic path $Q$ in colour~$1$ with limits $\Lambda=\Lambda(Q)=\Lambda(P)$, covering $A_1 \cup A_2 \cup A_3 \cup C \cup \Lambda$. 

Next, note that since $X \subseteq A_2$, it follows by choice of $C$ that $|X \setminus N(b,{\neq}1)| = |X \cap N(b,1)| < \kappa = |X|$ for every vertex $b$ in $B \setminus ( Q \cup B_4)$. Therefore, we may apply Lemma~\ref{lemma: covering bipartition classes} to the bipartite graph $G[B\setminus (Q\sqcup B_4), X]$ with the edges in colour $\neq 1$ to obtain disjoint monochromatic paths $P_{r+1},\dots, P_{2r-1}$ with different colours in $[2,r]$ covering $B\setminus (Q\sqcup B_4)$. 

Let $P_1$ be the path that results by deleting $X' = X \cap \left(\bigcup_{i=r+1}^{2r-1} P_i \right)$ from $Q$, using that $Q$ is $X$-robust. Together with the paths $P_2,\ldots,P_r$ provided by \ref{item4}, we have found a partition of $G$ into $2r-1$ disjoint monochromatic paths $P_1,\ldots,P_{2r-1}$ using every colour at most twice.

To complete the proof, it remains to construct the partitions $\mathcal{A}$ and $\mathcal{B}$. 

\begin{claim} 
There are disjoint sets $A_1,A_2\in [A]^\kappa$ and $B_1,\tilde{B}_2\in [B]^\kappa$ such that (up to renaming the colours) \begin{itemize}
\item $G[A_1,B_1]$ has a spanning path $P$ of order type $\kappa$ in colour $1$ all whose limits are contained in $B_1$ and  
\item $A_1\sqcup A_2$ is $\kappa$-star-linked in $\tilde{B}_2$ in colour~$1$.
\end{itemize}
\end{claim}
\begin{proof}\renewcommand{\qedsymbol}{$\Diamond$}
Apply Lemma~\ref{lemma: a path and a star-linked set} to the graph $G$ and the colouring $c$.
\end{proof}

\begin{claim} There is a partition $\{B_2,\tilde{B}_3\}$ of $\tilde{B}_2$ such that \begin{itemize}
\item $A_1\sqcup A_2$ is $\kappa$-star-linked in $B_2$ in colour $1$ and 
\item $G[A_1\sqcup A_2,\tilde{B}_3]$ has a perfect matching in colour $1$. 
\end{itemize}
\end{claim}

\begin{proof}\renewcommand{\qedsymbol}{$\Diamond$}
Take an increasing cofinal sequence of regular cardinals $(\kappa_\alpha)_{\alpha<\cf(\kappa)}$ in $\kappa$ and write $A_1\sqcup A_2$ as an increasing union of sets $\{A_\alpha \colon \alpha <\kappa\}$ with $|A_\alpha|=\kappa_\alpha$. Simultaneously define in $\cf{(\kappa)}$ many steps increasing sets $\{B_\alpha'\colon \alpha < \cf(\kappa)\}$ and $\{B_\alpha''\colon \alpha < \cf(\kappa)\}$ such that \begin{itemize}
\item $B_{\alpha}'$ and $B_{\alpha}''$ are disjoint subsets of $\tilde{B}_2$ with $|B_{\alpha}'| = |B_{\alpha}''| = \kappa_\alpha$ for $\alpha <\cf(\kappa)$, 
\item $G[A_\alpha,B_\alpha']$ has a perfect matching in colour $1$ and
\item $|N_G[F,1]\cap B_\alpha''|=\kappa_\alpha$ for every finite $F\subseteq A_1\sqcup A_2$. 
\end{itemize}
Letting $\tilde{B}_3:=\bigcup \{B_\alpha'\colon \alpha <\cf(\kappa)\}$ and $B_2:=\tilde{B}_2\backslash \tilde{B}_3$ completes the construction. 
\end{proof}

Let $A_3$ consist of those vertices in $A\setminus (A_1\sqcup A_2)$ that send $\kappa$ many edges in colour~$1$ to $\tilde{B_3}$. Note that since $A_2$ is $\kappa$-star linked in $B_2$ in colour $1$, it follows that $A_2\sqcup A_3$ is ${<}\kappa$-inseparable in $G[A_2\sqcup A_3,B_2\sqcup \tilde{B}_3]$ in colour $1$. 

\begin{claim}
There is a partition $\{\hat{B}_3,\tilde{B}_4\}$ of $\tilde{B}_3$ such that \begin{itemize}
\item $A_2\sqcup A_3$ is ${<}\kappa$-inseparable in $G[A_2\sqcup A_3,B_2\sqcup \hat{B}_3]$ in colour~$1$, and
\item $\tilde{B}_{4}$ has cardinality $\kappa$. 
\end{itemize}
\end{claim}

\begin{proof}\renewcommand{\qedsymbol}{$\Diamond$}
If $A_3$ is empty, then just take a balanced partition $\{\hat{B}_3,\tilde{B}_4\}$ of $\tilde{B}_3$. Otherwise, fix a sequence $(a_\alpha)_{\alpha<\kappa}$ of vertices in $A_3$ such that every vertex in $A_3$ appears $\kappa$ many times. Then fix a vertex in $N_G(a_\alpha,1)\cap \tilde{B}_3$ for $\hat{B}_3$ and another in $N_G(a_\alpha)\cap \tilde{B}_3$ for $\tilde{B}_4$ for every $\alpha<\kappa$ (all distinct). This can be done recursively in $\kappa$ many steps using that every vertex in $A_3$ sends $\kappa$ many edges in colour $1$ to $\tilde{B}_3$. It is easy to check that sets $\hat{B}_3$ and $\tilde{B}_4$ that arise in this manner fulfil our requirements. 
\end{proof}
The last partition class of $\mathcal{A}$ is just $A_4:=A\setminus (A_1\sqcup A_2\sqcup A_3)$. Applying Lemma~\ref{lemma: covering bipartition classes} to the spanning subgraph of $G[A_4,\tilde{B}_4]$ induced by the colours~$2,\dots, r$ (and the induced colouring) gives rise to disjoint monochromatic paths $P_2,\dots, P_r$ of different colours in $[2,r]$. Let $B_4:= \bigcup \{B \cap P_i\colon i\in [2,r]\}$ and $B_3:=B\setminus (B_1\sqcup B_2 \sqcup B_4)$.

We claim that $\mathcal{A}=\{A_1,\dots, A_4\}$ and $\mathcal{B}=\{B_1,\dots,B_4\}$ are as desired. Indeed, it is clear by construction that they are partitions of $A$ and $B$ respectively. From the first and second claim, it follows that \ref{item1} and \ref{item2} is satisfied respectively. By the definition of $A_4$ and $B_4$ in the last paragraph, we have \ref{item4}. And by the third claim, since $B_3 \supseteq \hat B_3$, it follows that \ref{item3} holds.
\end{proof}

Finally, we demonstrate that our approach for the proof of Theorem \ref{thm: main result}, which itself of course relies in many ways on ideas and results from Soukup's \cite{S17}, can be used to give a conceptually simple closing argument for a proof of Theorem \ref{thm: decomposition thm for complete graph}:

\CompleteCase*

\begin{proof}
Let $\kappa$ be an infinite cardinal, $G$ the complete graph on $\kappa$ and $c\colon E(G)\rightarrow [r]$ a colouring for some $r\ge 1$. 
Since the countable case has been solved in \cite[Theorem 2]{R}, we may assume that $\kappa$ is uncountable. 
Fix a partition $\{A,B\}$ of $V(G)$ such that both partition classes have cardinality $\kappa$. Apply Lemma \ref{lemma: a path and a star-linked set}---to the graph $G[A,B]$ and the colouring induced by $c$---in order to get disjoint sets $A_1,A_2\in [A]^\kappa$, $B_1,B_2\in [B]^\kappa$ and a path $P$ as in the lemma. Let $\Lambda$ be the set of limits of $P$ and write $A':= A_1\sqcup A_2$, $B':= V(G)\backslash A'$. Furthermore, let $C$ consist of $A'$ together with all those vertices in $V(G)\setminus (A'\cup \Lambda)$ that send $\kappa$ many edges in colour~$1$ to $A_2$. Apply Lemma~\ref{lemma: constructing robust paths}---to the graph induced by the edges of colour $1$ in $G[A',B']$ and the set $C$---in order to find a set $X\in [A_2]^\kappa$ and an $X$-robust path $Q$ as in the lemma. Next, apply Lemma~\ref{lemma: covering bipartition classes}---to the graph induced by the edges of colour $\neq 1$ in $G[X, B'\backslash Q]$ and the colouring induced by $c$---in order to find paths $P_2,\dots,P_r$ of different colours in $[2,r]$. The last path in colour $1$ is $Q \setminus \bigcup_i P_i$, which is a path due to the $X$-robustness of $Q$.   
\end{proof}


\section{Monochromatic paths covering a $\kappa$-star-linked set}
\label{sec3}
In this section, we prove Lemma \ref{lemma: a path and a star-linked set}. A partial result of Soukup's \cite{S17} will assist us: It implies that an edge-coloured complete bipartite graph with bipartition classes $(A,B)$ both of cardinality $\kappa>\aleph_0$ contains a monochromatic path $P$ of order type $\kappa$ covering a large ${<}\kappa$-inseparable subset of $A$ (cf. \cite[Theorem 5.10]{S17}). By modifying the proof, we obtain a strengthened version where this ${<}\kappa$-inseparable subset is even $\kappa$-star-linked, Theorem~\ref{thm: find large path} below. As the main result of this section, we explain how to establish Lemma~\ref{lemma: a path and a star-linked set} as a consequence of Theorem~\ref{thm: find large path}. The detailed proof of Theorem~\ref{thm: find large path} we defer until the end of this paper.

\subsection{Finding a monochromatic path covering a $\kappa$-star-linked set} First we remind the reader of a few concepts from Soukup's \cite{S17}:
Let $\kappa$ be a cardinal. Then $H_{\kappa,\kappa}$ denotes the graph $(\kappa \times \{0\}\cup \kappa \times \{1\}, E)$ where $$\text{$\{(\alpha,i),(\beta,j) \}\in E$ iff $i=1$, $j=2$ and $\alpha<\beta<\kappa$}$$ (cf. \cite[p.250, l.10--13]{S17}). Furthermore, a graph $G=(V,E)$ is of \emph{type $H_{\kappa,\kappa}$} if there are (not necessarily disjoint) subsets $A,B \subset V$ with $V = A \cup B$, and enumerations $A=\{a_\xi\colon  \xi <\kappa \}$ and $B=\{b_\xi\colon  \xi <\kappa \}$ such that $$\text{$\{a,b\}\in E(G)$ if $a=a_\xi$, $b= b_\zeta$ for some $\xi\le \zeta <\kappa$}.$$
The vertex set $A$ is called the \emph{main class} of $G$ and $B$ is called the \emph{second class} of~$G$ (cf. \cite[Definition 5.3]{S17}). Informally speaking, a type $H_{\kappa,\kappa}$ graph is just a copy of $H_{\kappa,\kappa}$ where the bipartition classes are allowed to intersect. 

Another concept that we need is those of a \emph{concentrated path} \cite[Definition 4.1]{S17}:
Let $G$ be a graph and $A\subseteq V(G)$. A path $P\subseteq G$ is concentrated on $A$ if and only if $$N(v)\cap A\cap V(P\restriction [x,v))\neq \emptyset $$ for all $v \in \Lambda(P)$ and $x\prec_P v$. 

\begin{theorem}\label{thm: find large path}  
If $G$ is an $r$-edge coloured graph of type $H_{\kappa,\kappa}$ with main class $A$, then there is a colour $k \in [r]$ and $X \in [A]^\kappa$ which is $\kappa$-star-linked in colour $k$, such that $X$ is covered by a monochromatic path of size $\kappa$ in colour $k$ concentrated on $X$.
\end{theorem}

\begin{proof}
Theorem~\ref{thm: find large path} follows from Theorem~\ref{thm_maintheorem111} which is a strengthening of \cite[Theorem~5.10]{S17}, to be proved in our last Section~\ref{sec_final} below.
\end{proof}

\subsection{Finding a monochromatic path covering an improved $\kappa$-star-linked set (proof of Lemma ~\ref{lemma: a path and a star-linked set})}
We need two more lemmas before we can prove Lemma~\ref{lemma: a path and a star-linked set}.

\begin{lemma}\label{lemma: char. concentrated}
Let $G$ be a bipartite graph with bipartition classes $A, B$. A path $P\subseteq G$ is concentrated on $A$ if and only if all limits of $P$ are contained in $B$. \qed
\end{lemma}

\begin{lemma}[cf. \cite{S17}, Lemma 3.4]\label{lemma: ultrafilter}
Let $\kappa$ be an infinite cardinal, $G=(V,E)$ a graph and $A,B\subseteq V(G)$. Suppose $A$ is $\kappa$-star-linked in $B$. Moreover, let $c\colon E(G)\rightarrow [r]$ be a colouring of $G$ with $r\ge 1$ many colours. Then there is a partition $\{A_i\colon i\in[r]\}$ such that $A_i$ is $\kappa$-star-linked in $B$ in colour $i$ for every $i\in [r]$.
\end{lemma}

\begin{proof}
Take a uniform ultrafilter $\mathcal{U}$ on $B$ with $B\cap N[F]\in \mathcal{U}$ for every finite $F\subseteq A$ and write $A_i=\{v\in A\colon N(v,i)\in\mathcal{U}\}$. 
Then for $i\in [r]$ and $F\subseteq A_i$, we have $N[F,i]\cap B\in \mathcal{U}$ and thus $N[F,i]\cap B$ has cardinality $\kappa$.  
\end{proof}

We are now ready to provide the proof for Lemma~\ref{lemma: a path and a star-linked set} which we restate here for convenience of the reader.

\SecondMainLemma*

\begin{proof}
Fix a set $A'\in [A]^\kappa$ that is $\kappa$-star-linked in as many colours as possible and let $I$ be the set of those colours. By Lemma \ref{lemma: ultrafilter}, the set $I$ is non-empty and we may assume that colour $1$ is contained in $I$. 

\begin{claim}
There are disjoint sets $B_1',B_2'\subseteq B$ such that $A'$ is $\kappa$-star-linked in $B_1'$ and $B_2'$, in all colours in $I$.   
\end{claim} 

\begin{proof}\renewcommand{\qedsymbol}{$\Diamond$}
Let $(\kappa_\alpha)_{\alpha<\cf(\kappa)}$ be a strictly increasing cofinal sequence of regular cardinals in $\kappa$ and write $A'$ as an increasing union of sets $(A_\alpha)_{\alpha<\cf(\kappa)}$ such that $A_\alpha$ has cardinality $\kappa_\alpha$ for $\alpha <\cf(\kappa)$. The sets $B_1'$ and $B_2'$ can be constructed simultaneously in $\cf(\kappa)$ many steps. Indeed---using that $A'$ is $\kappa$-star-linked in all colours in $I$ and the fact that $A_\alpha$ has only $\kappa_\alpha$ many finite subsets---it is straightforward to find two increasing sequences $(B^1_{\alpha})_{\alpha <\cf(\kappa)}$ and $(B^2_{\alpha})_{\alpha <\cf(\kappa)}$ of subsets of $B$ such that for all $\alpha < \cf(\kappa)$ \begin{itemize}
\item $B^i_\alpha$ has cardinality $\kappa_\alpha$ for both $i=1,2$,
\item $B^1_\alpha$ and $B^2_\alpha$ are disjoint, and
\item $A_\alpha$ is $\kappa_\alpha$-star-linked in $B^i_\alpha$ for both $i=1,2$ in all colours in $I$.
\end{itemize}
Letting $B_1'=\bigcup \{B^1_{\alpha} \colon \alpha<\cf(\kappa)\}$ and $B_2'=\bigcup \{B^2_{\alpha} \colon \alpha<\cf(\kappa)\}$ completes our construction.
\end{proof}
Fix $B_1'$ and $B_2'$ as in the above claim and let $\{A_1',A_2'\}$ be a partition of $A'$ such that both partition classes have cardinality $\kappa$. Since $G[A_1',B_1']$ is complete bipartite, it is in particular of type $H_{\kappa,\kappa}$, and so we may apply Theorem \ref{thm: find large path} to $G[A_1',B_1']$ to find a colour $k\in[r]$ and $X\in [A_1']^\kappa$ which is $\kappa$-star-linked in colour $k$ such that $X$ is covered by a monochromatic path $P$ (say) of size $\kappa$ in colour $k$ concentrated on~$X$. 

By the maximality of $I$ we have $k\in I$ and we may assume $k=1$. Furthermore, by Lemma~\ref{lemma: char. concentrated} we know that all limits of $P$ are contained in $B_1'$. Hence, letting $A_1:= A\cap P$, $A_2:= A_2'$, $B_1:= B\cap P$ and $B_2:= B_2'$ completes the proof.  \end{proof}



\section{Constructing robust paths}
\label{sec4}
In this section we will prove Lemma~\ref{lemma: constructing robust paths}. There are two major steps: First, we show how to find a ray $R$ that is $\{x\}$-robust for a single vertex~$x$. Second, we will construct the path $Q$ as a concatenation of rays each including a copy of $R$. The set $X$ for which $Q$ is $X$-robust will be the set of vertices $x$ in the various copies of~$R$.

\subsection{Constructing countable robust paths}
Consider the one-way infinite ladder on the positive integers shown in Figure \ref{fig: ladder}. The well-order $\le$ on the positive integers together with this ladder then forms a (generalised) path $R$, and it is easy to see that $R$ is $\{2\}$-robust. Indeed, the graph $R'=R-\{2\}$ has the one-way infinite path $R'=1436587\dots$ as a spanning subgraph. Note that additionally, the first vertices of $R'$ and of $R$ coincide. 

\vspace{0,5cm}
\begin{figure}[h]
\begin{center}
\def\svgwidth{7cm}
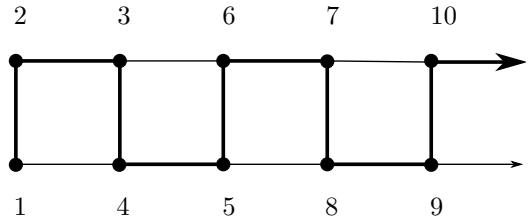
\end{center}
\vspace{0,5cm}
\caption{The fat edges indicate the path order of the one-way ladder on the positive integers.}
\label{fig: ladder}
\end{figure}

As we work in the bipartite setting, it is of importance that generalised paths that we want to install are bipartite. Our ray $R$ is bipartite as shown in Figure~\ref{fig: twisted ladder}.

\vspace{0,5cm}
\begin{figure}[h]
\begin{center}
\def\svgwidth{7cm}
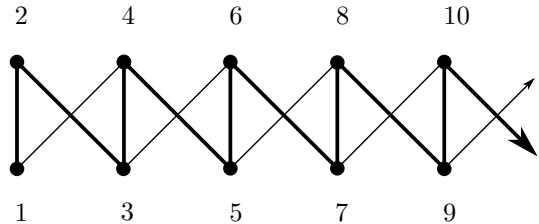
\end{center}
\vspace{0,5cm}
\caption{The top and bottom vertices in the ladder define the two bipartition classes of the one-way infinite ladder. The path order from Figure \ref{fig: ladder} is indicated fat again.}
\label{fig: twisted ladder}
\end{figure}
All countably infinite robust paths we construct will always consist of some finite path $Q$ followed by a copy of $R$, where we denote this concatenation by $Q^\frown R$. It is easy to see that the path $Q^{\frown}R$ is then $\{x\}$-robust, where $x$ is the vertex corresponding to the vertex $2\in V(R)$. The following lemma is our key lemma for constructing paths of that kind: 

\begin{lemma}\label{lemma: construct robust fragment}
Let $G$ be a bipartite graph with bipartition classes $A,B$ such that $A$ is countably infinite. Suppose further that $A$ is $\aleph_0$-star-linked and $a\in A$ is some fixed vertex. Then for any vertex $x\in A\setminus\{a\}$ there is an $x$-robust ray $R$ in $G$ starting in the vertex~$a$ and covering $A$. Moreover, there is a path order of $R-x$ with first vertex $a$. 
\end{lemma}

\begin{proof} Fix an enumeration of $(a_n)_{n\ge 1}$ of $A$ with $a_1= a$ and $a_2=x$. For $n\ge 1$ fix distinct vertices $(b_n)_{n\ge 1}$ such that $b_n$ is contained in the common neighbourhood of $\{a_n,a_{n+1},a_{n+2}\}$ for $n\ge 1$. This can be done since $A$ is $\aleph_0$-star-linked. 

Let us write $B'=\{b_n\colon n\ge 1\}$. Then $G[(A\setminus \{a_1\})\cup B']$ has a copy\footnote{The vertices in $A\setminus\{a_1\}$ correspond to the upper vertices in Figure \ref{fig: twisted ladder} and the vertices in $B'$ to the bottom vertices. The enumerations of $A\setminus \{a_1\}$ and respectively $B'$ are the enumeration which `go from left to the right'.} of the one-way infinite ladder $L$ on $\omega$ as a spanning subgraph where $b_1$ corresponds to the vertex $1$ and $x$ corresponds to the vertex $2$ of $L$. Let us write $R'$ for this copy of $L$ and endow $R'$ with the path order induced by the path order $\le$ on $L$. By our observations at the beginning of this subsection, the ray $R={a_1}^\frown R'$ is $\{x\}$-robust and starts with $a$. 
\end{proof}

\subsection{Constructing uncountable robust paths}

\ThirdMainLemma*

\begin{proof}
Let us write $\lambda_0$ for the first vertex on $P$ and let $\{\lambda_\alpha \colon 1\le\alpha <\kappa\}$ be the enumeration of the limits of $P$ along the path order of $P$, i.e. we have $\lambda_\alpha\prec_P \lambda_\beta$ whenever $1\le \alpha<\beta$. Fix an enumeration $\{c_\alpha\colon \alpha<\kappa\}$ of $C$, (choose $C=A_2$ if $C$ is not specified). Note that $C$ has indeed cardinality $\kappa$ as $A_2$ is included in $C$. We construct a sequence of pairwise disjoint paths $\mathcal{S}=(S_\alpha)_{\alpha<\kappa}$ and a sequence of distinct vertices $(x_\alpha)_{\alpha<\kappa}$ from $A_2$ satisfying the following: 
\begin{enumerate}
\item $S_\alpha$ has order type $\omega$,
\item $S_\alpha$ has first vertex $\lambda_\alpha$ and doesn't meet any other limits of $P$, 
\item $S_\alpha$ is $x_\alpha$-robust and there is a path order $\prec_{S_\alpha-x_\alpha}$ of $S_\alpha - x_\alpha$ that has first vertex $\lambda_\alpha$,
\item $S_\alpha\cap A\cap P=A\cap P\restriction [\lambda_\alpha,\lambda_\alpha + \omega)$ and 
\item $\bigcup_{\beta \leq \alpha} S_\alpha$ contains $c_\alpha$.  
\end{enumerate}
Once $\mathcal{S}$ is defined we obtain $Q$ as the concatenation $Q= S_0^\frown S_1 ^\frown S_2 ^\frown\cdots$ (formally, the path order is given by the lexicographic order on $\bigcup_{\alpha < \kappa} \{\alpha\} \times S_\alpha$). Indeed, conditions (1) and (2) guarantee that the limits of $Q$ and the limits of $P$ coincide, and so it follows from (4) that $Q$ is indeed a generalised path. By condition (5), the path $Q$ covers $C$. Finally, put $X=\{x_\alpha \colon \alpha<\kappa\}$.

\begin{claim} 
The path $Q$ is $X$-robust. 
\end{claim}
\begin{proof}\renewcommand{\qedsymbol}{$\Diamond$} In order to see that $Q$ is $X$-robust, let $X'\subseteq X$ be arbitrary. Let $S_\alpha'$ be the path $(S_\alpha-x_\alpha,\prec_{S_\alpha-x_\alpha})$, if $S_\alpha$ meets $X'$ and $S_\alpha'=S_\alpha$ otherwise. Then $Q'= S_0'^\frown S_1'^\frown S_2'^\frown\cdots$ is a path of order type $\kappa$ covering $Q-X'$.  
\end{proof}

It remains to define $\mathcal{S}=(S_\alpha)_{\alpha<\kappa}$. Suppose that $S_\alpha$ has already been defined for $\alpha<\beta$. Write $\Sigma_{\beta} := \bigcup_{\alpha < \beta } S_\alpha$, a set of cardinality ${<}\kappa$. We first find a finite path $T$ that \begin{itemize}
\item starts in $\lambda_\beta$,
\item ends in a vertex $a\in A_2$ (say),
\item contains $c_\beta$ (unless $c_\beta \in \Sigma_\beta$ already), 
\item avoids $\Sigma_{\beta}$ and meets $P$ only in $P\restriction [\lambda_\beta,\lambda_\beta+\omega)$.
\end{itemize}
\begin{claim}
A path $T$ as above exists.
\end{claim}
\begin{proof}\renewcommand{\qedsymbol}{$\Diamond$} Let $T_1$ be the path of (edge-)length $1$ or $0$, that starts in $\lambda_\beta$ and is followed by the successor of $\lambda_\beta$ on $P$ if $\lambda_\beta$ is not already contained\footnote{Since all limits of $P$ are contained in $B_1$, we have $\lambda_\beta\notin A_1$ as soon as $\beta\ge 1$. In the case where $\beta=0$, we might have $\lambda_\beta \in A_1$.} in~$A_1$. 

Let $w_1$ denote the last vertex on $T_1$, and note that $w_1 \notin \Sigma_\beta$ by (4). Since $A_1\sqcup A_2$ is $\kappa$-star-linked in $B_2$, we may chose any $w_2 \in A_2 \setminus \Sigma_\beta$ and find a vertex $w_3 \in (B_2 \cap N[\{w_1,w_2 \}]) \setminus \Sigma_\beta$ so that $w_1w_3w_2$ forms a path of (edge-)length two.

If $c_\beta$ is not yet covered by $\Sigma_\beta$, as $C$ is ${<}\kappa$-inseparable, we find a finite path $T_3$ that contains $c_\beta$, starts in the vertex $w_2$ and ends in a vertex $a \in A_2$ and avoids $$A_1 \cup \Lambda(P) \cup \Sigma_{\beta} \cup V( T_1)\cup V(T_2).$$  
Otherwise, we put $T_3 = \emptyset$ and $a = w_2$. Then $T$ can be chosen as $T_1^\frown T_2^\frown T_3$. 
\end{proof}
To complete the proof, we now find a path $R$ of order type $\omega$ such that it \begin{itemize}
\item starts in the vertex $a$ and avoids $T$ everywhere else,
\item is $\{x_\beta\}$-robust for a vertex $x_\beta \in A_2\setminus \{a\}$ and there is a path order of $R-x_\beta$ that starts with $a$,
\item avoids $\Sigma_{\beta}$ and meets $P$ only and in $P\restriction [\lambda_\beta,\lambda_\beta+\omega)$.
\end{itemize}

\begin{claim}
A path $R$ as above exists.
\end{claim}

\begin{proof}\renewcommand{\qedsymbol}{$\Diamond$}
Choose $x_\beta \in A_2 \setminus (\Sigma_\beta \cup V(T))$ arbitrary. Apply Lemma \ref{lemma: construct robust fragment} inside the bipartite graph $G[A',B']$ with the vertex $a$ and the vertex $x=x_\beta$, where $$A'=\{a,x_\beta\}\cup \left((A\cap P\restriction [\lambda_\beta, \lambda_\beta +\omega)) \setminus V(T) \right)$$ is countable, and $B'=B_2 \setminus (\Sigma_{\beta} \cup V(T))$.
\end{proof}
Letting $S_\beta = T^\frown R$ completes the construction of $S_\beta$ and thereby our proof is complete. 
\end{proof}

\section{A result extracted from Soukup's \cite{S17}}
\label{sec_final}

The following lemma of Soukup is the main tool of constructing large generalised paths. To state the lemma, we need the following definition.

\begin{definition}[{\cite[Definition~4.4]{S17}}]
Suppose that $G=(V,E)$ is a graph and $A \subseteq V$. We say that $A$ satisfies $\spadesuit_\kappa$ if for each $\lambda < \kappa$ there are $\kappa$ many disjoint paths concentrated on $A$ each of order type $\lambda$. 

Moreover, if we have a fixed edge-colouring $c \colon E \to [r]$ in mind, we write $\spadesuit_{\kappa,i}$ for ``$\spadesuit_\kappa$ in colour $i$".
\end{definition}

\begin{lemma}[{\cite[Lemma~4.6]{S17}}]
\label{lem_constructingpaths46}
Suppose that $G=(V,E)$ is a graph, $\kappa$ an infinite cardinal, and $A \in [V]^\kappa$. If
\begin{enumerate}
	\item $A$ is ${<}\kappa$-inseparable 
    
    and if $\kappa$ is uncountable, then    
    \item $A$ satisfies $\spadesuit_\kappa$, and
    \item there is a nice sequence of elementary submodels $(M_\alpha)_{\alpha<\text{cf}(\kappa)}$ for $\Set{A,G}$ covering $A$ so that there is $x_\beta \in A \setminus M_\beta$, $y_\beta \in V \setminus M_\beta$ with $x_\beta y_\beta \in E$ and
    \[|N_G(y_\beta) \cap A \cap M_\beta \setminus M_\alpha | \geq \omega \]
    for all $\alpha < \beta < \text{cf}(\kappa)$,
\end{enumerate}
then $A$ is covered by a generalised path $P$ concentrated on $A$.
\end{lemma}

Recall that Soukup considers for fixed $\kappa$ and $r\ge 1$ the following statements: 
\begin{itemize} \item[$(\textup{IH})_{\kappa,r}$] Let $H$ be a graph of type $H_{\kappa,\kappa}$ with main class $A$ and second class $B$ and $r\ge 1$. Then for every $r$-colouring of $H$, there is a colour $k$ and an $X \in [A]^\kappa$ so that $X$ satisfies all three conditions of Lemma~\ref{lem_constructingpaths46} in colour $k$. 
\item[$(\textup{IH})_{\kappa}$] The statement $(\textup{IH})_{\kappa,r}$ holds for every $r\ge 1$.
\end{itemize} 

Soukup's main result is then

\begin{theorem}[{\cite[Theorem~5.10]{S17}}]
\label{thm_souupstheorem}
$(\textup{IH})_{\kappa,r}$ holds for all $\kappa$. In particular, if $G$ is a graph of type $H_{\kappa,\kappa}$ with a finite-edge colouring, then we can find a monochromatic path of size $\kappa$ concentrated on the main class of $G$.
\end{theorem}

We now strengthen Soukup's results as follows, and consider the statements:

\begin{itemize} \item[$(\textup{IH})_{\kappa,r}'$] 
As the statement $(\textup{IH})_{\kappa,r}$ with the additional requirement that $X$ is also $\kappa$-star-linked in colour $k$.
\item[$(\textup{IH})_{\kappa}'$] The statement  $(\textup{IH})_{\kappa,r}'$ holds for every $r\ge 1$.
\end{itemize} 

The corresponding version of theorem \cite[Theorem~5.10]{S17} then reads:

\begin{theorem}
\label{thm_maintheorem111}
$(\textup{IH})_{\kappa}'$ holds for all $\kappa$. In particular, if $G$ is a graph of type $H_{\kappa,\kappa}$ with main class $A$ with an r-edge colouring, then there is a colour $k \in [r]$ and $X \in [A]^\kappa$ which is $\kappa$-star-linked in colour $k$, such that $X$ is covered by a monochromatic path of size $\kappa$ in colour $k$ concentrated on $X$.
\end{theorem}

The proof of Theorem~\ref{thm_maintheorem111} relies on the following lemma.

\begin{lemma}[{cf.\ \cite[Lemma 5.9]{S17}}]
\label{lem_5.9strengthened}
Let $\kappa$ be an infinite cardinal. Suppose that $c$ is an $r$-edge colouring of a graph $G=(V,E)$ of type $H_{\kappa,\kappa}$ with main class $A$ and second class $B$. Let $I \subsetneq [r]$, $X \in [A]^\kappa$ and suppose that $X$ is $\kappa$-star linked in all colours $i \in I$. If $(\textup{IH})_\lambda$ holds for all $\lambda < \kappa$ then either
\begin{enumerate}
	\item[(a)] there is an $i \in I$ such that $X$ satisfies $\spadesuit_{\kappa,i}$ , or
    \item[(b)] there is $\tilde{X} \in [X]^{{<}\kappa}$ and a partition $\set{X_j}:{j \in [r] \setminus I}$ of $X \setminus \tilde{X}$ such that $X_j$ is $\kappa$-star-linked in $B$ in colour $j$ for each $j \in [r] \setminus I$.
\end{enumerate}
\end{lemma}

\begin{proof}[Proof of Lemma~\ref{lem_5.9strengthened}]
Follow the proof of \cite[Lemma 5.9, p. 261]{S17} and in the last line apply the following Claim A instead of \cite[Claim 5.9.3]{S17}.

\bigskip

\textbf{Claim A} (cf.\ \cite[Claim 5.9.3]{S17})\textbf{.} \emph{Suppose that $c$ is an $r$-edge colouring of a graph $G=(V,E)$ of type $H_{\kappa,\kappa}$ with main class $A$ and second class $B$. Let $I \subset r$ and $X \subset A$. If for each finite subset $F \subset A$ we have
\[ |B \setminus \bigcup \set{N(x,i)}:{x \in F, i \in I} | = \kappa, \]
then there is a partition $\set{X_j}:{j \in r \setminus I}$ of $X$ such that $X_j$ is $\kappa$-star-linked in $B$ in colour $j$ for each $j \in [r] \setminus I$.
}

\begin{proof}\renewcommand{\qedsymbol}{$\Diamond$}
Take a uniform ultrafilter $U$ on $B$ so that $B \setminus \bigcup \set{N(x,i)}:{x \in F, i \in I} \in U$ for all finite subsets $F \subset A$. Define $X_j=\{x\in X\setminus \tilde{X}\colon N(x,j)\in U\}$ for each colour $j$ and note that $\set{X_j}:{j \in[r] \setminus I}$ partitions $X$. Since ultrafilters are closed under finite intersections, it follows that 
\[ N[F,j] \in U \]
for all finite subsets $F \subset X_j$ and $j \in [r] \setminus I$, and since the filter $U$ is uniform, we have $|N[F,j]| = \kappa$ and therefore that $X_j$ is $\kappa$-star-linked in $B$ for each such $j$.
\end{proof}

Indeed, by applying Claim A to the set $X \setminus \p{X^* \cup \tilde{A}}$ (defined in Soukup's proof), we readily obtain the stronger conclusion that the $X_j$ are not only ${<}\kappa$-inseparable, but even $\kappa$-star-linked.
\end{proof}

\begin{proof}[Proof of Theorem~\ref{thm_maintheorem111}]
We prove $(\textup{IH})_{\kappa,r}'$ by induction on $\kappa$ and $r$.

Note that $(\textup{IH})_{\omega}'$ holds by \cite[Lemma~3.4]{S17}, so we may suppose that $\kappa$ is uncountable. Also, $(\textup{IH})_{\kappa,1}'$ holds: From \cite[Observation~5.7]{S17}, we know that for any graph $G$ of type $H_{\kappa,\kappa}$, the main class of $G$ satisfies all conditions of Lemma~\ref{lem_constructingpaths46} (and so $(\textup{IH})_{\kappa,1}$ holds). However, it is clear that the main class $A$ is automatically $\kappa$-star-linked in $G$, and hence we have $(\textup{IH})_{\kappa,1}'$. 

Now fix an $r$-edge colouring of a graph $G$ of type $H_{\kappa,\kappa}$ with main class $A$ and second class $B$. As in Soukup's proof of {\cite[Theorem~5.10]{S17}} (Theorem~\ref{thm_souupstheorem} above), we may assume by the induction assumption $(\textup{IH})_{\lambda}'$ for $\lambda < \kappa$ that every $X \in [A]^\kappa$ satisfies condition $(3)$ in Lemma~\ref{lem_constructingpaths46}. 

Next, Soukup fixes a maximal $I\subseteq [r]$ with the property that there is a set $X\in [A]^\kappa$ such that $X$ is ${<}\kappa$-inseparable in all colours $i\in I$ and he fixes such $I$ and $X$. Instead, we now fix $I$ maximal with the property that there is a set $X\in [A]^\kappa$ such that $X$ is $\kappa$-star-linked in all colours $i\in I$. Then fix such $I$ and $X$. 

\bigskip

\textbf{Claim B} (cf.\ \cite[Claim 5.10.1]{S17})\textbf{.} \emph{There is $k \in I$ such that $\spadesuit_{\kappa,k}$ holds for $X$.
}

\begin{proof}\renewcommand{\qedsymbol}{$\Diamond$}
Suppose that $X$ fails $\spadesuit_{\kappa,i}$ for all $i \in I$. If $I \subsetneq [r]$, then apply Lemma~\ref{lem_5.9strengthened} in $G$ to the set $X$ and the set of colours $I$. As $X$ fails $X$ fails $\spadesuit_{\kappa,i}$ for all $i \in I$, condition $(b)$ of Lemma~\ref{lem_5.9strengthened} must hold (note that by induction assumption, $(\textup{IH})_{\lambda}'$ and hence $(\textup{IH})_{\lambda}$ hold for all $\lambda < \kappa$, so we may apply Lemma~\ref{lem_5.9strengthened}): However, this means there is a colour $j \in [r] \setminus I$ and a set $X_j \in [X]^\kappa$ such that $X_j$ is $\kappa$-star-linked in colour $j$ as well. But the fact that $X_j$ is then $\kappa$-star-linked in all colours $i \in I \cup \Set{j}$ contradicts the maximality of $I$.

Therefore, $I = [r]$ must hold. From this, however, we may obtain a contradiction precisely as in the second half of the proof of \cite[Claim 5.10.1]{S17}.
\end{proof}

Hence, the ``in particular'' part of the theorem, and hence Theorem~\ref{thm: find large path} follows by applying Lemma~\ref{lem_constructingpaths46} to the set $X$ provided by $(\textup{IH})_{\kappa}'$. The proof is complete.
\end{proof}


\begin{thebibliography}{99}
\bibitem{C} C. B\"{u}rger, Decomposing Edge-Coloured Infinite Graphs into Monochromatic Paths and Cycles, \emph{Master Thesis}, University of Hamburg (2017).
\bibitem{D} R. Diestel, Graph Theory (5th ed.), \emph{Springer Publishing Company} (2017).
\bibitem{ESSS17} M. Elekes, D. Soukup, L. Soukup and Z. Szentmikl\`{o}ssy, Decompositions of edge-coloured infinite complete graphs into monochromatic paths, \emph{Discrete Mathematics} \textbf{340} (2017), 2053--2069.
\bibitem{GG67} L. Gerencs\'{e}r, A. Gy\'{a}rf\'{a}s, On Ramsey-type problems \emph{Ann. Univ. Sci. Budapest. E\"{o}tv\"{o}s Sect. Math} \textbf{10} (1967), 167-170. 
\bibitem{GL73} A. Gy\'{a}rf\'{a}s and J. Lehel, A Ramsey-type problem in directed and bipartite graphs, \emph{Periodica Mathematica Hungarica} \textbf{3 (3-4)} (1973), 299-304 
\bibitem{P14} A. Pokrovskiy, Paritioning edge-coloured complete graphs into monochromatic cycles and paths, \emph{Journal of Combinatorial Theory, Series B} \textbf{106}, (2014) 70-97 
\bibitem{R} R. Rado, Monochromatic Paths in Graphs, \emph{Annals of Discrete Mathematics} \textbf{3} (1978), 191-194.
\bibitem{S15} D. Soukup, Colouring problems of Erd\H{o}s and Rado on infinite graphs, \emph{PhD Thesis}, University of Toronto (2015). 
\bibitem{S17} D. Soukup, Decompositions of edge-coloured infinite complete graphs into monochromatic paths II, \emph{Israel Journal of Mathematics} \textbf{221} (2017), 235-273.
\end{thebibliography}
\end{document}